\documentclass[12pt,a4paper]{amsart}
\usepackage{amsmath,amsfonts,amssymb,amsthm,amscd}
\usepackage{dsfont,mathrsfs,enumerate}
\usepackage{curves,epic,eepic}

\usepackage[colorlinks=true,
linkcolor=red, citecolor=blue, pagebackref]{hyperref}

\usepackage{url}

\usepackage[pdftex]{graphicx}
\usepackage{psfrag}
\usepackage{epic}
\usepackage{eepic}
\usepackage[matrix,arrow,curve]{xy}
\input{epsf}
\usepackage{tikz}
\usepackage{xcolor}
\usetikzlibrary{calc}
\usetikzlibrary{intersections}
\usetikzlibrary{through}
\usetikzlibrary{backgrounds}

\newtheorem{thm}{Theorem}[section]
\newtheorem{lemma}[thm]{Lemma}
\newtheorem{prop}[thm]{Proposition}
\newtheorem{coro}[thm]{Corollary}

\theoremstyle{definition}
\newtheorem{defi}[thm]{Definition}
\newtheorem{rem}[thm]{Remark}

\def\Z{\mathds Z}

\def\R{\mathds R}

\def\phi{\varphi}
\def\<{{\langle}}
\def\>{{\rangle}}

\newcommand{\area}[1]{\mathrm{area}(#1)}

\newcommand{\lw}[1]{\mathrm{lw}(#1)}
\newcommand{\lwd}[2]{\mathrm{lwd}_{#1}(#2)}
\newcommand{\interior}[1]{\mathrm{int}(#1)}
\newcommand{\floor}[1]{\left\lfloor #1 \right\rfloor}
\newcommand{\ceil}[1]{\left\lceil #1 \right\rceil}
\newcommand{\plvsl}[1]{\mathrm{plvsl}(#1)}
\newcommand{\conv}[1]{\mathrm{convhull}(#1)}

\begin{document}

\title[Area bounds for planar convex bodies]{
Area bounds for planar convex bodies containing a fixed number of interior integral points}

\author[Martin Bohnert]{Martin Bohnert}
\address{Mathematisches Institut, Universit\"at T\"ubingen,
	Auf der Morgenstelle 10, 72076 T\"ubingen, Germany}
\email{martin.bohnert@uni-tuebingen.de}
	
\begin{abstract}
We prove area bounds for planar convex bodies in terms of their number of interior integral points and their lattice width data. As an application, we obtain sharp area bounds for rational polygons with a fixed number of interior integral points depending on their denominator. For lattice polygons, we also present an equation for the area based on Noether's formula. 
\end{abstract}

\maketitle

\thispagestyle{empty}

\section{Introduction}

For any $d\in \Z_{\geq 1}$, we obtain from Minkowski's theorem \cite[p. 76]{Min10} the sharp volume bound $2^d$ for centrally symmetric convex bodies in $\R^d$ with a single interior integral point. Blichfeld \cite[10.]{Bli21} and van der Corput \cite{vdC36} noted that we can generalize Minkowski's bound to $2^{d-1}(k+1)$ if we have an arbitrary odd number $k$ of interior integral points.

However, if the convex body is not centrally symmetric, then there are already examples for arbitrarily large volumes for $d=2$. For example, for any $(k, l)\in \Z_{\geq 1}^2$ we already have in \cite[2.6.]{LZ91} the rational triangle
\begin{align*}
T_{k,l}:=\text{convhull}&\left((0,0),\left(1+\frac{1}{l},0\right), \left(0,(l+1)(k+1)\right)\right)\subseteq \R^2
\end{align*}
with $k=|\interior{T_{k,l}}\cap \Z^2|$ interior integral points and
\begin{align*}
\area{T_{k,l}}=\frac{(l+1)^2}{2l}(k+1)>l.
\end{align*}

However, it is possible to replace the central symmetry with other assumptions and still obtain a bound on the area. Examples of such modifications of Minkowski's theorem can be found in \cite[3.]{Sco88}.

Our goal is to use the concept of \textit{lattice width data} introduced in the next section to prove area bounds depending on this data.

\begin{thm}\label{TheoremBigWidth}
Let $K\subseteq \R^2$ be a convex body with $k:=|\interior{K}\cap \Z^2|\in \Z_{\geq 1}$. If $K$ has symmetric lattice width data or a lattice width $\lw{K}>5$, then we have
\begin{align*}
\area{K}\leq 2(k+1)-\floor{\max\left(\frac{\lw{K}}{2}-3,0\right)}\cdot \floor{\frac{\lw{K}}{2}-2}.
\end{align*}
Moreover, if there is a function $f:\R_{>0}\to \R_{>0}$ with $\area{K}\leq 2(k+1)-f(\lw{K})$ for every convex body $K$, then
\begin{align*}
\limsup_{x\to \infty} \frac{f(x)}{x^2}\leq \frac{3}{8}.
\end{align*}
\end{thm}

So in the case of large lattice width we can even sharpen the classical bound $2(k+1)$. But what can we say for smaller lattice widths, where the classical bound cannot be correct because of the existence of triangles $T_{k,l}$? The following theorem gives an answer.

\begin{thm}\label{TheoremSmallWidth}
Let $K\subseteq \R^2$ be a convex body with $k:=|\interior{K}\cap \Z^2|\in \Z_{\geq 1}$.\\
Then we can explicitly describe area maximizers in the case of $\lw{K}\leq 1$ for given lattice width data, and we have
\begin{align*}
\area{K}\leq \begin{cases} \frac{\lw{K}^2}{2(\lw{K}-1)}(k+1) & \text{if } \lw{K} \in (1,2]\\
 2(k+1)+\frac{1}{2} & \text{if } \lw{K}\in (2,5],
\end{cases}
\end{align*}
where the bounds for $\lw{K}\in (1,2]$ are sharp. 
\end{thm}

In particular, this theorem implies that the rational triangles $T_{k,l}$ are area maximizers for $\lw{K}=1+\frac{1}{l}$.\\

Probably polytopes with integral points as vertices, henceforth also called \textit{lattice polytopes} (for the lattice $\Z^d\subseteq \R^d$), are among the best-studied non-symmetric convex bodies that have finite volume bounds if they have a fixed number of interior integral points. This was first proved in \cite{Hen83} and improved volume bounds can be found in \cite[Theorem 1]{LZ91}, \cite[p. 17]{Pik01}, and \cite[Theorem 1.4.]{AKN20}.

However, these bounds do not seem to be sharp, since the largest known volumes for lattice polytopes with a fixed number of interior points from \cite{WZP82} are much smaller. Nevertheless, we do have sharp bounds for some special types of lattice polygons, e.g., for reflexive polytopes (see \cite[Corollary 1.2.]{BKN22}) and for lattice simplices with a single interior integral point (see \cite{AKN15}).

We also know from \cite[p. 1024]{LZ91} that, up to affine unimodular equivalence, there are only finitely many lattice polytopes in $\R^d$ whose volume is less than a fixed bound. Thus, we can also obtain sharp bounds through the classification of all lattice polytopes in $\R^d$ with a fixed number of interior points. This was done for lattice 3-polytopes with exactly $1$ or $2$ interior integral points in \cite{Kas10} and \cite{BK16}.

In \cite[Theorem 1]{LZ91} and \cite[p. 17]{Pik01} we also have volume bounds for lattice polytopes with exactly $k$ interior points in a sublattice $(l\Z)^d, l\in \Z_{\geq 2}$, which corresponds to the case of rational polytopes with denominator $l$. But also these bounds are far from the largest known examples in \cite[2.6.]{LZ91}, and it was even shown in \cite[§6]{Pik01} that the methods used are not good enough to show such small bounds.

There are also some classification results for these polytopes. In \cite[Theorem 4]{AKW17}, maximal half-integral lattice-free polygons - which is the case $d=2, k=0, l=2$ - were classified on the way of classifying all $\Z^3$-maximal lattice-free lattice 3-polytopes. In \cite[4.11.]{HHS25} we have the classification of almost $2$- and $3$-hollow LDP-polygons, a subcase of $d=2, k=1,$ $l\in \{2,3\}$, which was necessary for the work on $\varepsilon$-log canonical del Pezzo $\mathbb{K}^\ast$-surfaces.\\

It turns out that our methods are good enough to handle the case of rational polygons. As a corollary to the theorems above, we obtain in \ref{BoundRationalPolygons} that the triangle $T_{k,l}$ is indeed also an area maximizing rational polygon with denominator $l\in \Z_{\geq 2}$.

The case of lattice polygons is different. For lattice polygons $P_{k,1}\subseteq \R^2$ with $k$ interior integral points, there is an additional $+\frac{1}{2}$ and so
\begin{align*}
\area{P_{k,1}}\leq 2(k+1)+\frac{1}{2}
\end{align*}
and the convex hull $\conv{(0,0),(3,0),(0,3)}$ is the single area maximizer up to affine unimodular equivalence in this case, a result first proven combinatorially by Scott in \cite{Sco76}. However, there is also an area bound
\begin{align*}
\area{P_{k,1}}\leq 2(k+1)+2-\frac{n}{2}
\end{align*}
additionally depending on the number of vertices $n$, which was conjectured by Coleman in \cite[p. 54]{Col78} and then proved combinatorially in \cite{KO07}. In these cases, we write \textit{combinatorially proved} because there are older results from algebraic geometry that imply these results via the language of toric geometry. For example, Castryck showed in \cite[p. 506]{Cas12} that Coleman's conjecture was essentially already proven by Koelman in \cite[4.5.2.2]{Koe91}, and in \cite[2.3]{HS09} we see that Scott's volume bound corresponds to algebro-geometric results of del Pezzo and Jung.\\

In the final step, we combine two \textit{number 12-formula}, which originated from Noether's formula via toric geometry, to derive an equation for the area of a lattice polygon. This equation is presented in the following theorem, which not only implies Coleman's conjecture, but also helps us to handle the asymptotics in Theorem \ref{TheoremBigWidth}.

\begin{thm}\label{LatticePolygonArea}
Let $K\subseteq \R^2$ be a lattice polygon with $k:=|\interior{K}\cap \Z^2|\in \Z_{\geq 1}$, $n(\Delta_{P,smooth})$ the number of rays in the smooth refinement of the normal fan $\Delta_P$ of $P$, which we get from the rays generated by the boundary lattice points on the bounded edges of 
$\conv{\sigma \cap (\Z^2\setminus \{0\})}$ for each cone $\sigma$ of $\Delta_P$. Then
\begin{align*}
\area{P}=2(k+1)+2-\frac{n(\Delta_{P,smooth})}{2}-\area{\conv{\interior{P}\cap \Z^2)}}.
\end{align*}
\end{thm}

The article is organized as follows: In the second section, we introduce the concept of lattice width data for planar convex bodies. A central result is that we prove bounds for the vertical slicing lengths when we have $(1,0)$ as a lattice width direction. This also gives bounds for the number of interior integral points on integral vertical lines.

In the third section, we compute area bounds for planar convex bodies based on their  their lattice width data and the number of interior integral points. As a corollary, we obtain sharp area bounds for rational polygons, depending on the number of interior integral points and the denominator of the polygon. The proof of Theorem \ref{TheoremSmallWidth} is also provided at the end of the third section.

In the final section, we focus on lattice polygons. First, we provide a straightforward proof of Scott's area bound using the lattice width. Then, we use two variants of 'number 12'-theorems to prove Theorem \ref{LatticePolygonArea}, which allows us to estimate the error term of Scott's area bound in terms of the number of vertices and the lattice width. This also gives the asymptotic behavior of the bound in Theorem \ref{TheoremBigWidth}, which we can prove at the end of the last section.

\newpage

\section{Lattice width data for planar convex bodies}

In this section, we define the \textit{lattice width data}, which is a generalization of the concept of lattice width. By choosing appropriate coordinates that respect these lattice width data in \ref{LatticeWidthCoordinates}, we can control the vertical slicing lengths and the number of integral points on vertical lines of a planar convex body. These are the main results of this section in Theorem \ref{LengthVerticalSegments} and Corollary \ref{IntPointsVerticalSegments}, and we can use these results to obtain area bounds by integration later in the next section.

\begin{defi}
Let $K\subseteq \R^2$ be compact and convex with a nonempty interior. Then we call $K$ a planar convex body.
\end{defi}

\begin{defi}
Let $K\subseteq \R^2$ be a convex body. The \textit{width function of $K$} is defined by
\begin{align*}
\text{width}_K\colon (\R^2)^\ast\setminus \{0\} \to \R_{>0}, v\mapsto \max_{x\in K}v(x) - \min_{x\in K}v(x)
\end{align*}
and we call $\textit{width}_K(v)$ the \textit{width of $K$ in direction $v$}.
\end{defi}

\begin{rem}
The width function is a continuous function and therefore attains its infimum on any compact set, in particular on $S^1:=\{v\in (\mathbb{R}^2)^\ast \mid \|v\|=1 \}$. We can use this minimum to define a \textit{geometrical width} of $K$.
	
Since $\text{width}_K(\lambda v)=|\lambda|\text{width}_K(v)$ for all $\lambda \in \R\setminus \{0\}$, the width function also attains its infimum on $(\Z^2)^\ast\setminus \{0\}$ and so we can define a \textit{lattice width} depending on our lattice $\Z^2$ as follows.
\end{rem}

\begin{defi}
The \textit{lattice width of $K$} - denoted as $\lw{K}$ - is defined by
\begin{align*}
\lw{K}:=\min_{v\in (\Z^2)^\ast \setminus \{0\}} \text{width}_K(v). 
\end{align*}
A non-zero dual lattice vector $w \in (\Z^2)^\ast\setminus \{0\}$ is called a \textit{lattice width direction} of $K$, if $\text{width}_K(w)=\lw{K}$.
\end{defi}

\begin{rem}
If $w=(w_1,w_2)\in (\Z^2)^\ast$ is a lattice width direction of $K$, then $w$ is a primitive dual vector, i.e. $\gcd(w_1,w_2)=1$, and $-w$ is also a lattice width direction of $K$. By \cite[Theorem 1.1.]{DMN12} there are at most $8$ different lattice width directions for a planar convex body.
\end{rem}

\begin{lemma}\label{BrunnMinkowski}
Let $K\subseteq \R^2$ be a convex body, $\pi_1: \R^2\to \R, (x_1,x_2)\mapsto x_1$.\\
Then we have a concave function 
\begin{align*}
l_{K,\pi_1}: \pi_1(K) \to \R, t \mapsto | K \cap \{x_1=t\} |
\end{align*}
by measuring the length $|\cdot|$ of the vertical line segments. In particular, $l_{K,\pi_1}$ attains its maximum on a compact interval and is increasing before and decreasing afterwards.
\end{lemma}
\begin{proof}
If we intersect $K$ with a line through two points, one with a minimal and the other with a maximal first coordinate, we can decompose the area of $K$ into a convex top between the graph of a concave and a linear function, and a convex bottom between the graph of a linear function and a convex function. So we get the vertical line segment length function by adding two concave functions, and so $l_{K,\pi_1}$ is also concave.
\end{proof}

\begin{rem}
Using the Brunn-Minkowski inequality, we can get a similar result for the slice volumes of convex bodies of any dimension (e.g. \cite[12.2.1]{Mat02}).
\end{rem}

\begin{defi}
Let $K \subseteq \R^2$ be a convex body. We call the midpoint of the compact interval from Lemma \ref{BrunnMinkowski}, in which $K$ has the longest vertical slicing length, the \textit{position of the longest vertical slicing length} and write it as $\plvsl{K}$.
\end{defi}

\begin{rem}\label{AffineDiameter}
A vertical line segment of $K$ is of maximum length if and only if it is an affine diameter of $K$, i.e., there are parallel supporting lines of $K$ at the endpoints of the segment. For an overview of affine diameters, see \cite{Sol05}, where we can also find this characterization in \cite[3.1.]{Sol05}.
\end{rem}

\begin{lemma}\label{LatticeWidthCoordinates}
Let $K\subseteq \R^2$ be a convex body, $w\in (\Z^2)^\ast \setminus \{0\}$ a lattice width direction of $K$ and $\pi_1: \R^2\to \R, (x_1,x_2)\mapsto x_1$. Then there is an affine unimodular transformation $U_{A,b}$, i.e.
\begin{align*}
U_{A,b}: \R^2\to \R^2, x\mapsto Ax+b, \text{ with } A\in \mathbf{GL}(2,\Z), b\in \Z^2,
\end{align*}
so that the image $U_{A,b}(K)$ has a lattice width direction $wA^{-1}\in \{(-1,0),(1,0)\}$, we have $x_l,x_r\in \R$,  $0\leq x_l<1$ with $[x_l,x_r]=\pi_1(U_{A,b}(K))$ and 
\begin{align*}
2\cdot \plvsl{U_{A,b}(K)}\leq \ceil{x_l}+\floor{x_r}.
\end{align*}	
Moreover, the interval $[x_l,x_r]$ and $\plvsl{U_{A,b}(K)}$ are well-defined by $K$ and $w$, if we additionally ask for $\ceil{x_l}-x_l\leq x_r-\floor{x_r}$ in the case with
\begin{align*}
2\cdot \plvsl{U_{A,b}(K)}= \ceil{x_l}+\floor{x_r}.
\end{align*}
\end{lemma}
\begin{proof}
Since $w$ is a primitive dual vector, we can find another dual vector that complements $w$ to a lattice basis of $(\Z^2)^\ast$. By choosing $w$ as the first row and the other dual vector as the second row, we get a suitable unimodular matrix $A$. Now we have a unique first coordinate for $b$ to additionally get $0\leq x_l<1$.\\
If we were in the situation
\begin{align*}
2\cdot \plvsl{U_{A,b}(K)}> \ceil{x_l}+\floor{x_r}
\end{align*}
or
\begin{align*}
\ceil{x_l}-x_l> x_r-\floor{x_r} \text{ and } 2\cdot \plvsl{U_{A,b}(K)}= \ceil{x_l}+\floor{x_r},
\end{align*}
then we would choose $-A$ instead of $A$.

Note that the first row of $A$ and the first coordinate of $b$ are determined by the conditions, so both $[x_l,x_r]$ and $\plvsl{U(K)}$ are well-defined.
\end{proof}

\begin{defi}
Let $K\subseteq \R^2$ be a convex body.\\
If $w$ is a lattice width direction for $K$ and $U_{A,b}$, $[x_l,x_r]$ are as in \ref{LatticeWidthCoordinates}, then we define the \textit{lattice width data of $K$ for the lattice width direction $w$} as the pair
\begin{align*}
\lwd{w}{K} :=	([x_l,x_r], \plvsl{U_{A,b}(K)})
\end{align*}
and we call $\ceil{x_r}-\floor{x_l}-1$ the \textit{number of interior integral vertical lines for the lattice width direction $w$}.
\end{defi}

\begin{lemma}\label{LengthVerticalSegments}
Let $K\subseteq \R^2$ be a planar convex body, $(1,0)$ a lattice width direction of $K$, and $\pi_1(K)=[x_l,x_r]$.\\Then we have for the length $|\cdot |$ of the vertical line segments
\begin{align*}
|K\cap \{(x_1,x_2)\in \R^2 \mid x_1=h\}|\geq 
\begin{cases}
\min\left(\frac{x_r-x_l}{2},h-x_l\right) & \text{ if } x_l\leq h\leq  \plvsl{K}\\
\min\left(\frac{x_r-x_l}{2},x_r-h\right) & \text{ if } x_r\geq h\geq  \plvsl{K}.
\end{cases}
\end{align*}
\end{lemma}
\begin{proof}
It is enough to prove that
\begin{align*}
|K\cap \{(x_1,x_2)\in \R^2 \mid x_1=h\}|\geq h-x_l
\end{align*}
for
\begin{align*}
x_l < h < \min\left(\frac{x_l+x_r}{2} ,\plvsl{K}\right)
\end{align*}
because of the monotony of the function measuring the vertical slicing length on $[x_l, \plvsl{K}]$ by \ref{BrunnMinkowski} and the symmetric situation in the second case.
	
We assume that there is a $h_0$ with
\begin{align*}
x_l<h_0 < \min\left(\frac{x_l+x_r}{2} , \plvsl{K}\right)
\end{align*}
and 
\begin{align*}
|K\cap \{(x_1,x_2)\in \R^2 \mid x_1=h_0\}|< h_0-x_l.
\end{align*}
Now, we look at two supporting lines $t_a, t_b$ - thought as graphs of two affine linear functions $x_1\mapsto t_a(x_1), x_1\mapsto t_b(x_1)$ - supporting $K$ in $x_1=h_0$ from above in the point $(h_0,y_a):=(h_0,t_a(h_0))$ and from below in $(h_0,y_b):=(h_0,t_b(h_0))$.
	
Since an integral shearing $U_s, s\in \Z,$ in the direction of the $x_2$ axis, i.e., a unimodular transformation
\begin{align*}
U_s: \R^2\to \R^2, \begin{pmatrix}x_1\\x_2\end{pmatrix}\mapsto \begin{pmatrix}1&0\\s&1\end{pmatrix}\begin{pmatrix}x_1\\x_2\end{pmatrix},
\end{align*}
does not change the length of the vertical line segments and $(1,0)$ as a lattice width direction, we can assume without loss of generality that
\begin{align*}
0\leq y_a-t_a(h_0-1)<1.
\end{align*}
In particular, we can assume that $t_a$ is increasing.

Now we can look at the two cases 
\begin{align*}
y_b-t_b(h_0-1)\leq 0
\end{align*}
and 
\begin{align*}
0< y_b-t_b(h_0-1)<1,
\end{align*}
because the case $y_b-t_b(h_0-1)\geq 1$ is impossible due to $ y_a-t_a(h_0-1)<1$ and the monotonically increasing vertical length on $[x_l,\plvsl{K}]$ by \ref{BrunnMinkowski}.

In the first case, we have $y_b-t_b(h_0-1)\leq 0 \leq y_a-t_a(h_0-1)$ and therefore $t_b$ is decreasing and $t_a$ is increasing, so we get for $\pi_2\colon \R^2\to \R, (x_1,x_2)\mapsto x_2$ that
\begin{align*}
\text{width}_K((0,1))=&\max \pi_2(K)-\min \pi_2(K)\\
\leq&\frac{x_r-x_l}{h_0-x_l}\cdot |K\cap \{(x_1,x_2)\in \R^2 \mid x_1=h_0\}|\\
<&x_r-x_l\\
=&\text{width}_K((1,0)),
\end{align*}
so $(1,0)$ is not a width direction and we have a contradiction.\\
	
The second case, $0< y_b-t_b(h_0-1)<1$, is more complicated, because now $t_b$ increases as well. However, in the second case we can also assume that
\begin{align*}
y_a-t_a(h_0-1)\leq 1-(y_b-t_b(h_0-1)),
\end{align*}
because if this is not the case, we can use the integral shearing $U_{-1}$ defined above and a reflection on the $x_1$-axis to achieve this.

Thus, we obtain
\begin{align*}
&\text{width}_K((0,1))\\
=&\max \pi_2(K)-\min \pi_2(K)\\
\leq& (x_r-h_0)(y_a-t_a(h_0-1))+|K\cap \{(x_1,x_2)\in \R^2 \mid x_1=h_0\}|+\\
&\qquad (h_0-x_l)(y_b-t_b(h_0-1))\\
<& (x_r-h_0)(1-(y_b-t_b(h_0-1)))+h_0-x_l+ (h_0-x_l)(y_b-t_b(h_0-1))\\
=&x_r-x_l-(x_l+x_r)y_b-2h_0t_b(h_0-1)+(x_l+x_r)t_b(h_0-1)+2h_0y_b	\\	
=&x_r-x_l+(2h_0-(x_l+x_r))(y_b-t_b(h_0-1))\\
\leq&x_r-x_l\\
=&\text{width}_K((1,0)),
\end{align*}
where the last inequality follows from the fact that $2h_0\leq x_l+x_r$ by definition of $h_0$.

Thus, $(1,0)$ is not a width direction in this case either, and we have a contradiction.
\end{proof}

\begin{coro}\label{IntPointsVerticalSegments}
Let $K\subseteq \R^2$ be a planar convex body, $(1,0)$ a width direction of $K$, and $\pi_1(K)=[x_l,x_r]$.\\Then we have for the number of interior integral points on vertical line segments
\begin{align*}
&| \interior{K}\cap \{(x_1,x_2)\in \Z^2 \mid x_1=h\} |\\
\geq& \begin{cases}
\ceil{\min\left(\frac{x_r-x_l}{2},h-x_l\right)}-1 & \text{ if }  x_l\leq h\leq  \plvsl{K}, h\in \Z\\
\ceil{\min\left(\frac{x_r-x_l}{2},x_r-h\right)}-1 & \text{ if } x_r\geq h\geq  \plvsl{K}, h\in \Z.
\end{cases}
\end{align*}
\end{coro}
\begin{proof}
This follows directly from Theorem \ref{LengthVerticalSegments}, since a line segment
\begin{align*}
K\cap\{(x_1,x_2)\in \R^2 \mid x_1=h\}, h\in \Z
\end{align*}
with length greater than $k\in \Z_{\geq 1}$ has at least $k$ interior integral points.
\end{proof}

\section{Area bound for planar convex bodies with respect to their lattice width data}

In this section, we compute area bounds for convex bodies $K\subseteq \R^2$ with $k\geq 1$ interior integral points with respect to their lattice width data.

We use only elementary integration tools as methods. After choosing the coordinates appropriately, we approximate the area of the convex body using triangles and trapezoids to obtain our results.

The choice of coordinates was already done in \ref{LatticeWidthCoordinates}, where we essentially noticed that we can assume that $(1,0)$ is a lattice width direction of $K$. Moreover, we can assume the situations in the propositions of this section without loss of generality by \ref{LatticeWidthCoordinates}. For a suitable subdivision into triangles and trapezoids we use the vertical integral lines $x_1=h \in \Z$, supporting lines and chords of $K$.

\subsection{Planar convex bodies with one integral vertical line}

\begin{prop}\label{OneLine}
Let $(1,0)$ be a lattice width direction of $K$,  $0< a, b\leq 1$ with $\pi_1(K)=[1-a,1+b]$, $\plvsl{K}\leq 1$ and $k:=|\interior{K}\cap \Z^2|$ the number of interior integral points. Then
\begin{align*}
area(K)\leq
\begin{cases}
\frac{(a+b)^2}{2b}(k+1) & \text{ if } a> b\\
(a+b)(k+1)\leq 2(k+1) & \text{ if } a\leq b.
\end{cases}
\end{align*}
Furthermore, the bounds in the first inequality are sharp.
\end{prop}
\begin{proof}
We take two supporting lines of $K$ in $x_1=1$, one from above and one from below, and intersect the area between them with the strip $1-a\leq x_1\leq 1+b$ to get a triangle or trapezoid $T_K$ containing $K$. The dual vector $(1,0)$ is also a lattice width direction of $T_K$ since $\text{width}_K((1,0))=\text{width}_{T_K}((1,0))$ and $\text{width}_K(v)\leq\text{width}_{T_K}(v)$ for all $v\in (\Z^2)^\ast \setminus \{0\}$. We also have $\pi_1(T_K)=[1-a,1+b]$, $|\interior{T_K}\cap \Z^2|=k$ and $\plvsl{T_K}\leq 1$ by construction and so it is enough to look for sharp area bounds for triangles or trapezoids $T_K$.

The length of the line segment
\begin{align*}
T_K \cap \{(x_1,x_2)\in \R^2 \mid x_1=1\}
\end{align*}
is less than or equal to $k+1$ due to the $k$ interior integral points, so the area maximizers for $T_K$ are due to $\plvsl{T_K}\leq 1$ triangles with base length $\frac{a+b}{b}(k+1)$ and height $a+b$ in the case $a>b$ or parallelograms  with base length $k+1$ and height $a+b$ for $a\leq b$. Calculating their area yields the result.
\end{proof}

\begin{rem}\label{MaximizerLatticeWidth12}
In the Proposition, we not only gave sharp bounds, but we also explicitly described the area maximizers in the proof. In particular, we saw that for $\lw{K}\leq 1$, convex bodies with k interior integral points can be found with arbitrarily large area. However, we also found, that if the planar convex body has one integral vertical line and $\lw{K}>1$, then the area is bounded by
\begin{align*}
\frac{\lw{K}^2}{2(\lw{K}-1)}(k+1).
\end{align*}	
\end{rem}

\subsection{Planar convex bodies with two integral vertical lines}

\begin{prop}\label{TwoLines}
Let $(1,0)$ be a lattice width direction of $K$,  $0< a, b\leq 1$ with $\pi_1(K)=[1-a,2+b]$ and $
k_i:=|\interior{K}\cap \{x_1=i\}\cap \Z^2| \text{ for } i\in \{1,2\}, k:=k_1+k_2>0$.\\
If $\frac{3}{2}\geq \plvsl{K}> 1$, then
\begin{align*}
area(K)\leq & \begin{cases}
\frac{(a+b+1)^2(a(k_1+1)+b(k_2+1))}{2(a+b)^2} & \text{ if } a+b<1\\
(a+1)(k_1+1)+b(1-k_1+2k_2) & \text{ if } a+b\geq 1\end{cases}\\
\leq & \begin{cases}
\frac{(a+b+1)^2}{2(a+b)}(k+1) & \text{ if } a+b<1\\
2(k+1) & \text{ if } a+b\geq 1.		
\end{cases}.
\end{align*}
If $\plvsl{K}\leq 1$, then we have $l_1\geq l_2$ for $l_i:=|K\cap \{x_1=i\}|, i\in \{1,2\}$ and
\begin{align*}
area(K) \leq&\left(1+a-\frac{b^2}{2}\right)l_1+\left(\frac{b^2}{2}+b\right)l_2 \\
\leq &\begin{cases}
\left(\frac{1}{2}+a\right)k_1+\frac{3}{2}k_2+2+a & \text{ if } l_1\leq 2l_2\\
(1+a)k_1+\frac{k_2}{2}+\frac{3}{2}+a & \text{ if } l_1> 2l_2 \text{ and } k_2>0\\
2(k_1+1)+\frac{1}{2k_1}-(1-a)(k_1-1) & \text{ if } l_1> 2l_2 \text{ and } k_2=0
\end{cases}\\
\leq &\begin{cases}
2(k+1) & \textit{ if } k_2>0\\ 
2(k+1)+\frac{1}{2k} & \textit{ if } k_2=0.
\end{cases}
\end{align*}
\end{prop}
\begin{proof}
In the case $\frac{3}{2}\geq \plvsl{K}>1, a+b<1$, we can use the results for one integral width line from Proposition \ref{OneLine} twice, once for the part of $K$ with $1-a\leq x_1 \leq 1+\frac{a}{a+b}$ and a second time for the part with $2-\frac{b}{a+b}\leq x_1\leq 2+b$. Since $\frac{a}{a+b}>a$, $\frac{b}{a+b}>b$ and $\frac{3}{2}\geq \plvsl{K}>1$ we get
\begin{align*}
\area{K}\overset{\ref{OneLine}}{\leq} &\frac{\left(a+\frac{a}{a+b}\right)^2}{2a}(k_1+1)+\frac{\left(b+\frac{b}{a+b}\right)^2}{2b}(k_2+1)\\
= & \frac{(a+b+1)^2(a(k_1+1)+b(k_2+1))}{2(a+b)^2}\\
\leq & \frac{(a+b+1)^2}{2(a+b)}(k+1).
\end{align*}
In the case $\frac{3}{2}\geq \plvsl{K}>1, a+b\geq 1$ we can assume without loss of generality that $k_1\geq 1$ and use Proposition \ref{OneLine} for the parts $1-a\leq x_1\leq 2-b$ and $2-b\leq x_1\leq 2+b$. Since $a\geq 1-b$ and $\plvsl{K}>1$ we get  
\begin{align*}
\area{K}\overset{\ref{OneLine}}{\leq}& (a-b+1)(k_1+1)+2b(k_2+1)\\
\leq&(a+1)(k_1+1)+b(1-k_1+2k_2)\\
\overset{(\ast)}{\leq }&(a+1)(k_1+1)+2bk_2\\
\leq &2(k+1),
\end{align*}
where we have used $k_1\geq 1$ in $(\ast)$.

In the case $\plvsl{K}<1$ we have $l_1\geq l_2$ due to the monotonically decreasing vertical length on $[\plvsl{K},2]$ by \ref{BrunnMinkowski}. We bound $K$ by the union of two trapezoids, namely a first trapezoid bounded by $1-a\leq x_1\leq 2$, one supporting line in $x_1=1$ from above and one from below, and a second trapezoid bounded by $2\leq x_1\leq 2+b$, a line through the points of $K$ on $x_1=1$ and $x_1=2$ with maximal $x_2$ coordinate and a line through the points of $K$ on $x_1=1$ and $x_1=2$ with minimal $x_2$ coordinate.

The first trapezoid has an area less than $(1+a)l_1$ because $l_1$ is longer than the midsegment of the trapezoid. The second trapezoid has parallel sides of length $l_2$ and $l_2-b(l_1-l_2)$ and height $b$, so we get 
\begin{align*}
\area{K}\leq (1+a)l_1+\frac{b}{2}\left(l_2+(l_2-b(l_1-l_2))\right)=\left(1+a-\frac{b^2}{2}\right)l_1+\left(\frac{b^2}{2}+b\right)l_2.
\end{align*}

If $l_1\leq 2l_2$ we can enlarge the second trapezoid by taking $x_1\leq 3$ instead of $x_1\leq 2+b$, i.e. use $b=1$, and with $a\leq 1$ and $l_i\leq k_i+1$ we get  
\begin{align*}
\area{K}\leq \left(\frac{1}{2}+a\right)l_1+\frac{3}{2}l_2\leq \left(\frac{1}{2}+a\right)k_1+\frac{3}{2}k_2+2+a.
\end{align*}

If $l_1>2l_2$ we can enlarge the second trapezoid to the triangle $T$ bounded by $x_1=2$ and the lines through the extrema of $K$ at $x_1=1$ and $x_1=2$. This triangle has base length  $l_2$ and height $\frac{l_2}{l_1-l_2}$, so with $l_1>2l_2$ we get 
\begin{align*}
\area{T}=\frac{l_2^2}{2(l_1-l_2)}<\frac{l_2}{2}
\end{align*}
We can now use $l_i\leq k_i+1$ to get
\begin{align*}
\area{K}\leq (1+a)l_1+\area{T}\leq (1+a)l_1+\frac{l_2}{2}\leq (1+a)k_1+\frac{k_2}{2}+\frac{3}{2}+a.
\end{align*}

In the case $l_1>2l_2, k_2=0$ we can go even further. We can assume $l_1>\frac{16}{9}$, because for $l_1\leq \frac{16}{9}$ the given bound is correct since with $k\geq 1$ we have $k_1\geq 1$ and
\begin{align*}
2l_1+\frac{l_2}{2}\leq \frac{9}{4}l_1\leq 4\leq 2(k_1+1).
\end{align*}
Thus, assume $l_1>\frac{16}{9}$. Since $k_2=0$ we also have $l_2\leq 1$ and so the triangle $T$ has an area of
\begin{align*}
\area{T}=\frac{l_2^2}{2(l_1-l_2)}\leq\frac{1}{2(l_1-1)}.
\end{align*}
The function $x\mapsto x+\frac{1}{x}$ is increasing on $[1,\infty)$. Since $l_1> \frac{16}{9}$ we have $2(l_1-1)>1$ and so $2(l_1-1)+\frac{1}{2(l_1-1)}\leq 2k_1+\frac{1}{2k_1}$. Thus, we get with $k_1-1\leq l_1\leq k_1+1$ the bound
\begin{align*}
\area{K}\leq (1+a)l_1+\frac{1}{2(l_1-1)}\leq 2(k_1+1)+\frac{1}{2k_1}-(1-a)(k_1-1).
\end{align*}

The last inequality in the proposition follows directly in all cases except in the case $k_1=0, k_2=1$. But then $1\geq l_1\geq l_2$  and we get
\begin{align*}
\area{K}\leq 2l_1+\frac{3}{2}l_2 \leq 4=2(k+1).
\end{align*}
\end{proof}

\begin{rem}\label{MaximizerLatticeWidth12_b}
If we look in the case $\plvsl{K}\leq 1$ at those convex bodies with $a+b<1$ or equivalently $b<1-a$, we get
\begin{align*}
\area{K}\leq& \left(1+a\right)l_1-\frac{b^2}{2}(l_1-l_2)+bl_2\\
\leq&\left(1+a\right)l_1+(1-a)l_2\\
\leq & 2l_1\\
\leq &2 (k+1),
\end{align*}
Together with the case $\plvsl{K}>1, a+b<1$ and \ref{MaximizerLatticeWidth12} we get that for $\lw{K}\in (1,2]$ we have a sharp area bound and can even describe area maximizers.
\end{rem}

\subsection{Planar convex bodies with at least three integral vertical lines}

\begin{prop}\label{3ormore}
Let be $m\in \Z_{\geq 3}$, $(1,0)$ a lattice width direction of $K$,  $0< a, b\leq 1$ with $\pi_1(K)=[1-a,m+b]$, $\plvsl{K}\leq \frac{m+1}{2}$, $l_i:=|K\cap \{x_1=i\}|$ and $k_i:=|\interior{K}\cap \{x_1=i\}\cap \Z^2|$. Then
\begin{align*}
\area{K}\leq \begin{cases}
2k+2-\floor{\frac{m-5}{2}}\cdot \floor{\frac{m-3}{2}} & \text{if } k_m\neq 0 \text{ or } m\geq 5\\
2k+2+\frac{1}{2} & \text{else}.
\end{cases}
\end{align*}
\end{prop}

\begin{proof}
Let be $i_{\max}:=\lceil \plvsl{K}\rceil$. Then $1\leq i_{\max}\leq \lceil \frac{m+1}{2}\rceil<m$ and by definition of $\plvsl{K}$ for $i\in \Z$ we have
\begin{align*}
\area{K\cap \{(x_1,x_2)\in \R^2 \mid i\leq x_1\leq i+1\}}\leq \begin{cases}
l_{i+1} & \textit{if } i+1< i_{\max}\\
l_{i} & \textit{if } i\geq i_{\max}
\end{cases}
\end{align*}
and with the trapezoid defined by $i_{\max}-1\leq x_1\leq i_{\max}+1$ and supporting lines in $x_1=i_{\max}$ we get
\begin{align*}
\area{K\cap \{(x_1,x_2)\in \R^2 \mid i_{\max}-1\leq x_1\leq i_{\max}+1\}}\leq 2l_{i_{\max}}.
\end{align*}
If $k_1k_m\neq 0$, then we get from these bounds with $l_i\leq k_i+1$
\begin{align*}
\area{K}\leq & 2l_{i_{\max}}+\sum_{i=1, i\neq {i_{\max}}}^{m}l_i\\
\leq&2k_{i_{\max}}+2+\sum_{i=1, i\neq {i_{\max}}}^{m}(k_i+1)\\
\leq &k+k_{i_{\max}}+m+1\\
\overset{(\ast)}{\leq} &2k+2- 2\cdot \sum_{i=1}^{\floor{(m-5)/2}}i\\
\leq &2k+2- \floor{\frac{m-5}{2}}\cdot \floor{\frac{m-3}{2}},
\end{align*}
where we used in $(\ast)$ that we can count the interior integral points of $K$ by counting the points of $\interior{K}$ on $x_1=i_{\max}$, count at least one point on the other $m-1$ integral vertical lines (remember $k_1k_m\neq 0$ !) and also some more points on these lines according to \ref{IntPointsVerticalSegments} (at least one point more for each line till $\plvsl{K}$). So we get
\begin{align*}
k_{i_{\max}}+(m-1)+2\cdot \sum_{i=1}^{\floor{(m-5)/2}}i\leq k.
\end{align*}

For the above calculation, we need $k_1k_m\neq 0$ to have at least one point on each integral width line. But we will now see that everything still works fine if we only have $k_m\neq 0, k_1=0$, because we can avoid the $k_1+1$ term of 
\begin{align*}
\area{K\cap \{(x_1,x_2)\in \R^2 \mid 1-a\leq x_1\leq 1\}}
\end{align*}
with the following argument divided in cases with respect to $\plvsl{K}$.

If $\plvsl{K}\in (1,2]$, then we have $i_{\max}=2$ and with $k_2\geq 1$ by \ref{IntPointsVerticalSegments} we get
\begin{align*}
&\area{K\cap \{(x_1,x_2)\in \R^2 \mid 1-a\leq x_1\leq 3\}}\\
\leq& \area{K\cap \{(x_1,x_2)\in \R^2 \mid 1-a\leq x_1\leq 1+a\}}+\\
 &\qquad\area{K\cap \{(x_1,x_2)\in \R^2 \mid 1+a\leq x_1\leq 3\}}\\
\leq & 2a + (2-a)l_2 \leq 2a+(2-a)(k_2+1)\leq 2(k_{i_{\max}}+1)
\end{align*}
and if $\plvsl{K}>2$ we get with $k_2\geq 1$
\begin{align*}
&\area{K\cap \{(x_1,x_2)\in \R^2 \mid 1-a\leq x_1\leq 2\}}\\
\leq& \area{K\cap \{(x_1,x_2)\in \R^2 \mid 1-a\leq x_1\leq 1+a\}}+\\
 &\qquad\area{K\cap \{(x_1,x_2)\in \R^2 \mid 1+a\leq x_1\leq 2\}}\\
\leq & 2a + (1-a)l_2 \leq 2a+(1-a)(k_2+1) \leq k_2+1
\end{align*}
and the case $\plvsl{K}\leq 1$ is not possible, because of $k_1=0$ and \ref{IntPointsVerticalSegments}.\\

We can use similar arguments as in the case $\plvsl{K}>2$ above to avoid the $k_m+1$ term in the case $k_m=0$ when $m\geq 5$, because then we get
\begin{align*}
\area{K\cap \{(x_1,x_2)\in \R^2 \mid m-1\leq x_1\leq m+b\}}\leq k_{m-1}+1.
\end{align*}
But if $m<5$, we cannot use this argument, because then $i_{\max}-1\leq x_1 \leq i_{\max}+1$ and $m-1\leq x_1$ can have non-empty intersection. So if $k_m=0$ and $m<5$, then in the cases $m=3, i_{\max}=2$ or $m=4, i_{\max}=3$ we have to use the last inequality of \ref{TwoLines} to get
\begin{align*}
&\area{K\cap \{(x_1,x_2)\in \R^2 \mid m-2\leq x_1\leq m+b\}}\leq 2k_{i_{\max}}+2+\frac{1}{2},
\end{align*}
i.e. we get an additional $+\frac{1}{2}$ compared to the calculations above.
\end{proof}

\begin{proof}[Proof of Theorem \ref{TheoremSmallWidth}] Combining the Propositions \ref{OneLine}, \ref{TwoLines}, \ref{3ormore} and the Remarks \ref{MaximizerLatticeWidth12}, \ref{MaximizerLatticeWidth12_b} we get the result.
\end{proof}

\subsection{Consequences for area bounds of some specific planar convex bodies}

\begin{coro}\label{SymmetricLWD}
Let $K\subseteq \R^2$ be a convex body with $\lw{K}>5$ or with symmetric lattice width data, i.e., there is an odd number $n$, $0<a\leq 1$ and a lattice width direction $w$ with $\lwd{w}{K}=([1-a,n+a], \frac{n+1}{2})$. Then $\area{K}\leq 2(k+1)$.
\end{coro}
\begin{proof}
If $\lw{K}>5$, then we have at least $5$ integral width lines and the result follows with \ref{3ormore}.

For symmetric lattice width data, we also have an odd number of integral width lines. Therefore, the claim follows with \ref{OneLine} for one integral vertical line and with \ref{3ormore} for $5$ and more integral vertical lines. For three integral vertical lines and symmetric lattice width data we get $\plvsl{K}=2$, so we can bound $K\cap \{1\leq x_1\leq 3\}$ not only by a trapezoid but even by a parallelogram due to \ref{AffineDiameter}. Thus, we can avoid an additional $+\frac{1}{2}$ similar to the proof of \ref{3ormore}.
\end{proof}

\begin{defi}
Let $P\subseteq \R^2$ be a polygon, whose vertices have rational coordinates. Then we call $P$ a \textit{rational polygon} and the \textit{denominator of $P$} is the smallest $l\in \Z_{\geq 1}$, so that $lP$ has integral vertices.
\end{defi}

\begin{coro}\label{BoundRationalPolygons}
Let $P\subseteq \R^2$ be a rational polygon with denominator $l\in \Z_{\geq 2}$ and $k>1$ interior integral points. Then we have the sharp bound
\begin{align*}
\area{P}\leq \area{T_{k,l}}=\frac{(l+1)^2}{2l} (k+1).
\end{align*}
\end{coro}
\begin{proof}
This is the bound from \ref{OneLine}, and all the other bounds that we have proven are smaller.
\end{proof}

\begin{rem}
In particular, Corollary \ref{BoundRationalPolygons} shows that the triangles $T_{k,l}$ from the introduction and already described in \cite[2.6.]{LZ91} are examples of area-maximizing rational polygons with $k$ interior integral points and denominator $l$. The other area maximizers for $l>2$ can be constructed from $T_{k,l}$ by a shearing in the direction of the $x_2$ axis. For $l=2, k=1$ there could be even more maximizers, since the bound in \ref{BoundRationalPolygons} is in this case $\frac{9}{2}=2(k+1)+\frac{1}{2}$.

Note also that \ref{BoundRationalPolygons} is incorrect for $l=1, k=1$, because for the standard triangle $\Delta_2:= \conv{(0,0),(1,0),(0,1)}$ we have $\area{3\cdot \Delta_2}=\frac{9}{2}>4$. 
\end{rem}

\begin{rem}
Rational polygons with one interior integral point are also important in toric geometry. For IP-polygons, i.e., lattice polygons with 0 as an interior point, and LDP-polygons, i.e., IP-polygons with primitive lattice points as vertices, we have the \textit{order} $o_p$ of such a polygon $P$ as
\begin{align*}
o_P:=\min \{k\in \Z_{\geq 1} \mid \interior{P/k}\cap \Z^2= \{(0,0)\}\}.
\end{align*}
In \cite[4.6.]{KKN10} there is the open question whether there is a bound on the area of IP-polygons that is cubic in $o_P$. We get the positive answer from \ref{BoundRationalPolygons}, that the sharp bound is given by
\begin{align*}
\area{P}\leq o_P(o_P+1)^2.
\end{align*}
This also implies better bounds relative to the maximal local index of the polygon than the bounds given in \cite[4.5.]{KKN10}.
\end{rem}

\section{Area bounds for lattice polygons}

The methods, we used for planar convex bodies, work even better for lattice polygons. In particular the following lemma leads directly to the classification of hollow lattice polygons (e.g. done in  \cite[4.1.2]{Koe91}) or to a direct simple proof of Scott's area bound, which also gives an interesting proof of Pick's theorem en passant.

On the other hand we can use results from algebraic geometry via the language of toric geometry to obtain area bounds for lattice polygons. This approach provides an estimate for the error term in Scott's area bound on this way, implying the correct asymptotic behavior in Theorem \ref{TheoremBigWidth}.

\begin{defi}
The \textit{standard lattice triangle} is given as
\begin{align*}
\Delta_2:=\conv{(0,0),(1,0),(0,1)}.
\end{align*}
\end{defi}

\begin{rem}
Lattice polygons that are affine unimodular equivalent to integer multiples of $\Delta_2$ are often special cases in the results of this section. This starts already with the following lemma.
\end{rem}

\begin{lemma}\label{numpointslp}
Let $P\subseteq \R^2$ be a lattice polygon with lattice width direction $(1,0)$ and lattice width $\lw{P}\geq 2$ which is not affine unimodular equivalent to an integer multiple of the standard lattice triangle $\Delta_2$. Then every vertical integral line $x_1=i\in \Z$, which has non-empty intersection with $\interior{P}$, contains at least one interior lattice point of $P$.
\end{lemma}
\begin{proof}
By \ref{IntPointsVerticalSegments} we know, that the intersection of a vertical integral line with the interior of the lattice polygon $P$ is either empty or has a length greater than or equal to $1$, where the length $1$ is only possible for those vertical integral lines with minimal or maximal $x_1$ coordinate among them with non-empty intersection. But if, without loss of generality, the integral vertical line with minimal $x_1$ coordinate among them with non-empty intersection has length $1$ and contains no interior lattice point of $P$, then we have after a translation and a shearing in the direction of the $x_2$-axis (both of them do not change the lattice width direction $(1,0)$) the points $(0,0), (1,0)$ as boundary points at the bottom and $(0,1)$ as boundary point at the top of the lattice polygon. Therefore, we get an integer multiple of $\Delta_2$ because $(1,0)$ is a lattice width direction, which contradicts the assumptions.
\end{proof}

\begin{coro}
Let $P\subseteq \R^2$ be a lattice polygon.\\
Then $|\interior{P}\cap \Z^2|=0$ if and only if $P$ has $\lw{P}=1$ or $P\cong 2\cdot\Delta_2$.
\end{coro}
\begin{proof}
From Lemma \ref{numpointslp} we have that $P$ has $\lw{P}=1$ or $P$ is unimodular equivalent to an integer multiple of $\Delta_2$. All $m\cdot \Delta_2$ with $2\neq m\in  \Z_{\geq 1}$ have interior lattice points or lattice width $1$, so we get $m=2$, if $\lw{P}\neq 1$.
\end{proof}

\begin{thm}[\cite{Sco76}]\label{Scott} Let $P\subseteq \R^2$ be a lattice polygon with $k\geq 1$ interior integral points. Then
\begin{align*}
\area{P}	\leq 2(k+1)+\frac{1}{2}
\end{align*}
and equality holds if and only if $P \cong 3\cdot \Delta_2$.
\end{thm}

\begin{proof} We chose \textit{lattice width coordinates} introduced in \ref{LatticeWidthCoordinates}, i.e. we can assume without loss of generality, that $P\subseteq \{(x_1,x_2)\in \R^2 :0\leq x_1\leq \lw{P}\}$. We also define $l_i:=|P \cap \{x_1=i\}|, k_i:=|\interior{P}\cap \{x_1=i\}\cap \Z^2|$ and as length of line segments
\begin{align*}
r_{i,a}:=&|\{(x_1,x_2)\in \R^2 : x_1=i,\\& \floor{\max\{j\in \R : (i,j)\in P\}} \leq x_2\leq \max\{j\in \R : (i,j)\in P\}\}|,\\
r_{i,b}:=&|\{(x_1,x_2)\in \R^2 : x_1=i,\\& \ceil{\min\{j\in \R : (i,j)\in P\}} \geq x_2\geq \min\{j\in \R : (i,j)\in P\}\}|.
\end{align*}
for all $i\in [0,\lw{P}]\cap \Z$.
	
We start with a triangulation of $P$ in special (non-lattice!) triangles with one of the segments $P\cap \{x_1=i\}$ as base, height $1$ and another edge at the boundary of $P$. Using the triangulation in the first step, the definition of the $r_{i,a}, r_{i,b}, k_i$ in the second and $0\leq r_{i,a}, r_{i,b}\leq 1$ in the third, we get
\begin{align*}
\area{P}=&\frac{l_0}{2}+\frac{l_{\lw{P}}}{2}+2\cdot \sum_{i=1}^{\lw{P}-1}\frac{l_i}{2}\\
=&\frac{l_0}{2}+\frac{l_{\lw{P}}}{2}+\sum_{i=1}^{\lw{P}-1}\left(k_i-1+r_{i,a}+r_{i,b}\right)\\
=&\frac{l_0}{2}+\frac{l_{\lw{P}}}{2}+k-(\lw{P}-1)+\sum_{i=1}^{\lw{P}-1}\left(r_{i,a}+r_{i,b}\right)\\
\leq&\frac{l_0}{2}+\frac{l_{\lw{P}}}{2}+k-(\lw{P}-1)+2(\lw{P}-1)\\
=&\frac{l_0}{2}+\frac{l_{\lw{P}}}{2}+k+\lw{P}-1.
\end{align*}
Since $P$ is a lattice polygon, its longest possible slicing length is equal to one of the $l_i$, which we call $l_{i_{\max}}$, and this is longer than the midsegment of the trapezoid
\begin{align*}
\conv{P\cap \{x_1=0\}, P\cap \{x_1=\lw{P}\}}\subseteq P,
\end{align*}
which has the length $\frac{l_0+l_{\lw{P}}}{2}$. If $P$ is not affine unimodular equivalent to an integer multiple of $\Delta_2$, then $k_i\geq 1$ for all $1\leq i\leq \lw{P}-1$ by \ref{numpointslp} and in this case we get
\begin{align*}
area(P)\leq\frac{l_0}{2}+\frac{l_{\lw{P}}}{2}+k+\lw{P}-1\leq k_{i_{\max}}+1+k+1+\sum_{i=1, i\neq i_{\max}}^{\lw{P}-1}k_i
= 2k+2.
\end{align*}
If $P$ is unimodular equivalent to $\lw{P}\cdot \Delta_2$, then we have
\begin{align*}
\area{P}=\frac{\lw{P}^2}{2}, \quad |\interior{P} \cap \Z^2|=\sum_{i=1}^{\lw{P}-2}i=\frac{(\lw{P}-1)(\lw{P}-2)}{2}.
\end{align*}
So we have $\area{3\cdot \Delta_2}=2(k+1)+\frac{1}{2}$ and  $\area{P}=2k-\frac{\lw{P}^2}{2}+3\lw{P}-2\leq 2(k+1)$ for  $\lw{P}\geq 4$.
\end{proof}

\begin{rem}
We can also get the result in \ref{Scott} as a corollary of the area bounds in Proposition \ref{OneLine}, \ref{TwoLines} and \ref{3ormore}.
\end{rem}

\begin{rem}
If we look at the details of our proof of \ref{Scott}, then we see that we have also almost given a proof of \textit{Pick's formula}
\begin{align*}
\area{P}=k+\frac{b}{2}-1
\end{align*}
from \cite[§. I.]{Pic99} for a lattice polygon $P$ with $k$ interior lattice points and $b$ boundary lattice points. This is because $r_{i,a}=1$ or $r_{i,b}=1$ if and only if we have a boundary lattice point of $P$ on $x_1=i$ and the other $r_{i,a}, r_{i,b}$ add up to $1$ in pairs, because $\Z^2$ is centrally symmetric to the center of any line segment between two lattice points. If we use $b_i$ for the number of boundary lattice points of $P$ at $\{x_1=i\}$, $b_a$ for their number at the upper boundary, $b_b$ for their number at the lower boundary below and $b$ for the number of all boundary lattice points of $P$, then we get Pick's formula via
\begin{align*}
area(P)=&\frac{l_0}{2}+\frac{l_{\lw{P}}}{2}+k-(\lw{P}-1)+\sum_{i=1}^{\lw{P}-1}\left(r_{i,a}+r_{i,b}\right)\\
=&k+\frac{b_0-1+b_{l+1}-1}{2}-(\lw{P}-1)+\sum_{i=1}^{\lw{P}-1}\left(r_{i,a}+r_{i,b}\right)\\
=&k+\frac{b_0+b_{l+1}}{2}-\lw{P}+b_a+b_b+\sum_{i=1, r_{i,a}<1}^{\lw{P}-1}r_{i,a}+\sum_{i=1, r_{i,b}<1}^{\lw{P}-1}r_{i,b}\\
=&k+\frac{b_0+b_{l+1}}{2}-\lw{P}+b_a+b_b+\frac{\lw{P}-1-b_a}{2}+\frac{\lw{P}-1-b_b}{2}\\
=&k+\frac{b}{2}-1.
\end{align*}
We can even do this without choosing width coordinates first, so we just need our \textit{integration method}. Note again that we have triangulated in non-lattice (!) triangles, as opposed to the triangulations in the classical proofs of Pick's theorem (see e.g. \cite[13.]{AZ18}).
\end{rem}

\begin{rem}
For lattice polygons, we can also use results from algebraic geometry via the language of toric varieties. Max Noether's formula
\begin{align*}
\chi(\mathcal{O}_X)=\frac{K_X\cdot K_X + e(X)}{12},
\end{align*}
which is the topological part in the surface case of the Hirzebruch-Riemann-Roch theorem, implies two combinatorial 'number 12'-theorems in the toric setting, which were essentially already part of the classical textbooks of toric geometry. These combinatorial 'number 12'-theorems lead to a formula for the area of a lattice polygon, which not only implies Coelman's conjecture but also gives a description of the error term. Parts of this story were already mentioned in \cite{Cas12}, but the explicit area formula was not written down there.
\end{rem}

\begin{lemma}\label{number12}
Let $P\subseteq \R^2$ be a lattice polygon, $F(P):=\conv{(\text{int}(P)\cap \Z^2}$ the convex hull of its interior lattice points.\\
If $\dim(F(P))=0$, then we have for the numbers of boundary lattice points $b(P)$ and $b(P^\ast)$ of $P$ and the dual polygon $P^\ast=\{v\in (\R^2)^\ast \mid v(x)\geq -1 \ \forall x\in P\}$
\begin{align*}
b(P)+b(P^\ast)=12.
\end{align*}
If $\dim(F(P))=2$, then we have for the lattice points on the boundary
\begin{align*}
b(P)-b(F(P))=12-n(\Delta_{P,smooth}),
\end{align*}
where $n(\Delta_{P,smooth})$ is the number of rays in the smooth refinement of the normal fan $\Delta_P$ of $P$ that we get from the rays generated by the boundary lattice points on the bounded edges of 
\begin{align*}
\conv{\sigma \cap (\Z^2\setminus \{0\})}
\end{align*}
for each cone $\sigma$ of $\Delta_P$.
\end{lemma}
\begin{proof}
Both formulas are consequences of Noether's formula in the case of smooth complete toric surfaces. For the first formula see \cite{PR00}, for the second formula see \cite[4.5.2]{Koe91}. For the background of the formulas in the classical textbooks see \cite[p. 57 f.]{Oda78}, \cite[p. 45 f.]{Oda88}, \cite[p.44]{Ful93} and \cite[Theorem 10.5.10.]{CLS11}.
\end{proof}

\begin{proof}[Proof of Theorem \ref{LatticePolygonArea}]
If $\dim(F(P))=0$, then we have $\area{F(P)}=0$ and also $n(\Delta_{P,smooth})=b(P^\ast)$. So we get with Pick's formula, $k=1$ and the result from the first formula in \ref{number12}
\begin{align*}
\area{P}=&k+\frac{b}{2}-1=1+\frac{12-n(\Delta_{P,smooth})}{2}-1\\
=&\underbrace{2(k+1)+2}_{=6}-\frac{n(\Delta_{P,smooth})}{2}-\underbrace{\area{F(P)}}_{=0}.
\end{align*}

If $\dim(F(P))=1$, then we also have $\area{F(P)}=0$. The formula is correct for a rectangle $\conv{(0,-1),(0,1),(n,-1),(n,1)}$ for any $n\in \Z_{\geq 3}$. By \cite[4.3]{Koe91} we can obtain a lattice polygon affine unimodular equivalent to $P$ by intersecting such a rectangle with up to four half-planes, whose normal vectors have $\pm 1$ or $\pm 2$ as first coordinate. We see that the equation remains correct for each intersection, because the loss of area is compensated by the loss of interior integral points and the change in $n(\Delta_{P,smooth})$.

If $\dim(F(P))=2$ we have from Pick's theorem for the number of boundary and interior lattice points $b(P), k(P)$ and $b(F(P)), k(F(P))$ of $P$ and $F(P)$ due to $k(P)=k(F(P))+b(F(P))$
\begin{align*}
\area{P}+\area{F(P)}=&k(P)+\frac{b(P)}{2}-1+k(F(P))+\frac{b(F(P))}{2}-1\\
=&2k(P)+\frac{b(P)-b(F(P))}{2}-2.
\end{align*}
Using the second formula from \ref{number12}, we get the result.
\end{proof}

Let $n$ be the number of vertices of the lattice polygon. We can use Theorem \ref{LatticePolygonArea} to give a proof of Coleman's conjecture $\area{P}\leq 2(k+1)+2-\frac{n}{2}$ from \cite{Col78}, because $n(\Delta_{P,smooth})$ is the number of rays in a refinement of the normal fan that has $n$ vertices, and so we get $n_{\Delta_{P,smooth}}\geq n$. We also get the additional correction term $\area{F(P)}$. To have the correction term depending on the lattice width, we can use the classical result $\area{P}\geq \frac{3}{8}\lw{P}^2$ from \cite{FM75} for the polygon $F(P)$ together with $\lw{F(P)}=\lw{P}-2$ from \cite[Theorem 13]{LS11} for polygons not affine unimodular equivalent to a multiple of $\Delta_2$. So we have the following corollary.

\begin{coro}
Let $P\subseteq \R^2$ be a lattice polygon with $\lw{P}\geq 3$ and not affine unimodular equivalent to an integer multiple of  $\Delta_2$.\\
Then we have
\begin{align*}
\area{P}\leq &2k+4-\frac{n}{2}-\frac{3}{8}(\lw{P}-2)^2.
\end{align*}
\end{coro}

\begin{rem}\label{AreaMinimumSharp}
Because for integer multiples of $\Delta_2$ we have 
\begin{align*}
\area{P}=2k-\frac{\lw{P}^2}{2}+3\lw{P}-2,
\end{align*}
the corollary holds for all lattice polygons with $\lw{P}\geq 10$.

For large $\lw{P}$ the inequality is also sharp up to a constant less than or equal to $3$, since the inequality $\area{P}\geq \frac{3}{8}\lw{P}^2$ is sharp for integer multiples of $Q:=\conv{(1,0),(0,1),(-1,1)}$, we have $F(mQ)=(m-1)Q$ for $m\in\Z_{\geq 1}$ and $n(\Delta_{Q,smooth})=9$.
\end{rem}

\begin{proof}[Proof of Theorem \ref{TheoremBigWidth}]
If $\lw{K}>5$ we have at least five integral vertical lines and thus the area bound follows from \ref{3ormore} and \ref{SymmetricLWD}.

For the asymptotic behavior of the correction term $f$ in $\lw{K}$ to the classical bound $2(k+1)$, we must have
\begin{align*}
\limsup_{x\to \infty} \frac{f(x)}{x^2}\leq \frac{3}{8},
\end{align*}
because of the example $mQ$ in \ref{AreaMinimumSharp}.
\end{proof}

\vspace{15mm}

\textbf{Acknowledgements.} I would like to thank Helena Stegmaier and Andreas Bäuerle for their helpful comments. I would also like to thank Victor Batyrev, who suggested investigating generalizations of Scott's inequality, and Daniel Hättig whose work \cite{Hae22} motivated me to investigate area bounds of rational polygons.

\end{document}